\documentclass[11pt,letterpaper]{amsart}

\usepackage{geometry}
\usepackage[english]{babel}
\usepackage{amsmath,amssymb,mathtools,enumerate,latexsym,graphicx}
\usepackage{xcolor,etoolbox}
\usepackage{hyperref}
\hypersetup{
  colorlinks=true,
  linkcolor=black,
  citecolor=black,
  urlcolor=black,
  pdftitle={Holomorphic expanding maps},
  pdfauthor={Jiesong Zhang}
}

\newtheorem{theorem}{Theorem}[section]
\newtheorem{lemma}[theorem]{Lemma}

\theoremstyle{remark}

\newtheorem{mainthm}{\bfseries Theorem}

\newcommand{\wt}{\widetilde}

\title[Holomorphic expanding maps]
{Holomorphic expanding maps}

\author{Jiesong Zhang}
\address{Department of Mathematics\\
Kungliga Tekniska h\"ogskolan, Lindstedtsv\"agen 25\\
SE-100 44 Stockholm\\
Sweden}
\email{jiesongz@kth.se}

\begin{document}

\begin{abstract}
We prove that every holomorphic expanding map on a compact connected complex manifold is
biholomorphically conjugate to an affine expanding endomorphism of a complex
infra-nilmanifold.
\end{abstract}

\maketitle

\section{Introduction}

Let $M$ be a compact connected smooth manifold without boundary. A $C^1$ map $f\colon M\to M$ is called
\emph{expanding} if there exist a Riemannian metric on $M$ and a constant
$\lambda>1$ such that
\[
  \|Df_x(v)\|\geq \lambda\|v\|,
  \quad
  \forall x\in M, v\in T_xM.
\]
An expanding map is necessarily a covering map.

The Shub--Franks--Gromov classification states that every expanding map is topologically conjugate to an affine expanding endomorphism of an infra-nilmanifold \cite{shub_1969,franks_1970,gromov_1981}; see also \cite[Theorem~4.5]{dekimpe_2012} for the precise affine formulation.
The conjugacy is H\"older continuous but is, in general, not a $C^1$
diffeomorphism: a $C^1$ diffeomorphism conjugating the maps identifies the
eigenvalues of the derivatives of corresponding periodic return maps, whereas
these eigenvalues can be changed by arbitrarily small local $C^\infty$
perturbations. Moreover, the smooth manifold $M$ need not be diffeomorphic to its infra-nilmanifold
model \cite{farrell_jones_1978}. Additional assumptions are therefore needed to obtain a $C^1$
or smoother conjugacy; see, for example,
\cite{shub_sullivan_1985,fang_2012,spatzier_yang_2017,gogolev_rodriguez_hertz_2023}.

The dynamics of holomorphic maps on complex manifolds are often more rigid
than those of general differentiable maps, owing to the compatibility between
the dynamics and the complex structure. In this note, we confirm this
phenomenon for holomorphic expanding maps.

\begin{mainthm}\label{theorem expanding}
Let $f\colon X\to X$ be a holomorphic expanding map on a compact connected complex manifold $X$. Then there exist a complex infra-nilmanifold $N$, a holomorphic affine expanding endomorphism $\alpha\colon N\to N$, and a biholomorphism $h\colon X\to N$ such that
\[
  h\circ f=\alpha\circ h.
\]
\end{mainthm}
By a complex infra-nilmanifold we mean an infra-nilmanifold
$\Pi\backslash G$ endowed with the complex structure induced by a
$\Pi$-invariant left-invariant integrable complex structure on $G$.

The proof combines three elementary observations. First, all deck
transformations of the universal cover $\wt X$ are biholomorphic. Second, after an affine change of coordinates, the algebraic model of $f$ is induced by an expanding automorphism $A$ of a simply connected nilpotent Lie group $G$ with $A(\Gamma)\subset\Gamma$, where $\Gamma$ is the translation lattice, and
\[
  \Gamma_\infty:=\bigcup_{n\geq0}A^{-n}(\Gamma)
\]
is dense in $G$. This produces a dense family of biholomorphic translations on
$\wt X$. Third, a locally uniform limit of holomorphic maps is holomorphic. The
closure of this family therefore recovers all left translations of $G$.
The Bochner--Montgomery theorem \cite{bochner_montgomery_1945} makes the resulting action smooth, allowing the conjugacy to transport the complex structure to a left-invariant one on $G$.

\section{Proof of Theorem \ref{theorem expanding}}
Let $X$ and $f$ satisfy the assumptions of Theorem~\ref{theorem expanding}. By the Shub--Franks--Gromov classification \cite{shub_1969,franks_1970,gromov_1981}, there exist an infra-nilmanifold $\Pi\backslash G$, an affine
expanding endomorphism $\alpha\colon\Pi\backslash G\to\Pi\backslash G$, and a homeomorphism $h\colon X\to\Pi\backslash G$ such that
\[
  h\circ f=\alpha\circ h.
\]

Choose an affine lift of $\alpha$ to $G$. It has a unique fixed point. Conjugating both the lift and $\Pi$ by a left translation sending this fixed point to the identity, and modifying $h$ accordingly, we may assume that $\alpha$ is induced by an expanding Lie-group automorphism $A\colon G\to G$ satisfying
\[
    A\Pi A^{-1}\subseteq\Pi.
\]
Let $\Gamma=\Pi\cap G$ be the translation lattice. Then
\[
    A(\Gamma)\subseteq\Gamma.
\]
Let $F\colon\wt X\to\wt X$ and $H\colon\wt X\to G$ be compatible lifts of $f$ and $h$, respectively. Then $F$ is biholomorphic and
\begin{equation}\label{equation expanding equivariance}
  H\circ F=A\circ H.
\end{equation}

\begin{lemma}\label{lemma asymptotic lattice}
The subgroup
\[
  \Gamma_\infty:=\bigcup_{n\geq0}A^{-n}(\Gamma)
\]
is dense in $G$.
\end{lemma}

\begin{proof}
Choose a compact set $K\subset G$ such that $G=\Gamma K$. For any $g\in G$
and $n\geq1$, write
\[
  A^n(g)=\gamma_nk_n,
  \qquad \gamma_n\in\Gamma,\quad k_n\in K.
\]
It follows that
\[
  A^{-n}(\gamma_n)=gA^{-n}(k_n^{-1}).
\]
Since $A^{-1}$ is contracting, $A^{-n}(k_n^{-1})\to e$ and $A^{-n}(\gamma_n)\to g$.
\end{proof}

For $g\in G$, define a homeomorphism $\Phi_g\colon\wt X\to\wt X$ by
\[
  \Phi_g:=H^{-1}\circ L_g\circ H,
\]
where $L_g$ denotes left translation by $g$ on $G$.

\begin{lemma}\label{lemma full translations expanding}
The map $\Phi_g$ is biholomorphic for every $g\in G$.
\end{lemma}

\begin{proof}
We first consider the map $\Phi_\delta$ with $\delta=A^{-n}(\gamma)\in\Gamma_\infty$. For
$\gamma\in\Gamma$, the map
$\Phi_\gamma=H^{-1}\circ L_\gamma\circ H$ is a deck transformation and hence
is biholomorphic. Moreover,
\[
  L_{\delta}=L_{A^{-n}(\gamma)}=A^{-n}\circ L_\gamma\circ A^n.
\]
By \eqref{equation expanding equivariance},
\[
  \Phi_\delta=H^{-1}\circ L_\delta\circ H=F^{-n}\circ\Phi_\gamma\circ F^n,
\]
so $\Phi_\delta$ is biholomorphic.

Now fix $g\in G$. By Lemma~\ref{lemma asymptotic lattice}, choose $\delta_j\in\Gamma_\infty$ with $\delta_j\to g$. Then $L_{\delta_j}\to L_g$ locally uniformly on $G$, and hence
$\Phi_{\delta_j}\to\Phi_g$ locally uniformly on $\wt X$. Thus $\Phi_g$ is
holomorphic. Likewise,
$\Phi_{\delta_j}^{-1}=\Phi_{\delta_j^{-1}}$ converges locally uniformly to
$\Phi_{g^{-1}}=\Phi_g^{-1}$, and $\Phi_g$ is biholomorphic.
\end{proof}

The maps $\{\Phi_g:g\in G\}$ define a continuous, free, and transitive action
of $G$ on $\wt X$, and every element acts biholomorphically. We then use the following consequence of the Bochner--Montgomery theorem \cite{bochner_montgomery_1945}.

\begin{theorem}[Bochner--Montgomery]\label{theorem BM}
Let $L$ be a finite-dimensional Lie group acting continuously and effectively
on a smooth manifold $M$ by $C^\infty$ diffeomorphisms. Then the action map $L\times M\to M$ is smooth.
\end{theorem}
By Theorem~\ref{theorem BM}, the action given by $\{\Phi_g:g\in G\}$ is smooth. Choose $x_0\in\wt X$ with
$H(x_0)=e$ and consider the orbit map
\[
  \Theta\colon G\to\wt X,
  \qquad
  g\longmapsto\Phi_g(x_0)=H^{-1}(g).
\]
The kernel of $D_e\Theta$ is the Lie algebra of the stabilizer of $x_0$, which is zero because the action is free. Since $\dim_{\mathbb R}G=\dim_{\mathbb R}\widetilde X$, the orbit map is a local diffeomorphism. It is bijective by freeness and transitivity, and hence is a global $C^\infty$ diffeomorphism. Consequently, $H=\Theta^{-1}$ is a $C^\infty$ diffeomorphism.

Push the complex structure of $\wt X$ forward to $G$ by $H$ and denote it by
$J_G$. This complex structure is smooth and integrable. It follows from the definition that 
\[
  H\colon\wt X\longrightarrow(G,J_G)
\]
is a biholomorphism. Since every $\Phi_g$ is biholomorphic, every left translation $L_g$ is $J_G$-holomorphic. Hence $J_G$ is left-invariant. Moreover, $A=H F H^{-1}$ is $J_G$-holomorphic. 

Finally, since $H$ lifts $h$, it conjugates the full deck group of $\wt X\to X$ to $\Pi$. Every deck transformation of $\wt X$ is
biholomorphic, so the affine action of $\Pi$ on $(G,J_G)$ is biholomorphic. Hence $J_G$ descends to $N=\Pi\backslash G$, and $H$ descends to a biholomorphism $h\colon X\to N$. The map $A$ induces the holomorphic affine expanding endomorphism $\alpha$, and $h\circ f=\alpha\circ h$. This proves Theorem~A.

\bibliographystyle{amsalpha}
\bibliography{ref}
\end{document}